\RequirePackage{fix-cm}
\documentclass[smallcondensed]{svjour3}     
\smartqed  
\usepackage{graphicx}
\usepackage{amsfonts,amsmath,amssymb,float,bm}
\newcommand{\scal}[2]{\langle #1,#2\rangle}

\begin{document}

\title{On constructions and properties of self-dual generalized bent functions
\thanks{
	The work is supported by Mathematical Center in Akademgorodok under agreement No. 075-15-2019-1613 with the Ministry of Science and Higher Education of the Russian Federation and Laboratory of Cryptography JetBrains Research.
	}
}


\author{Aleksandr Kutsenko}


\institute{A. Kutsenko \at
	Sobolev Institute of Mathematics, Novosibirsk \\
	Novosibirsk State University, Novosibirsk, Russia\\
	\email{alexandrkutsenko@bk.ru}           
}

\date{Received: date / Accepted: date}

\maketitle

\begin{abstract}
Bent functions of the form $\mathbb{F}_2^n\rightarrow\mathbb{Z}_q$, where $q\geqslant2$ is a positive integer, are known as generalized bent (gbent) functions. Gbent functions for which it is possible to define a dual gbent function are called regular. A regular gbent function is said to be self-dual if it coincides with its dual. In this paper we explore self-dual gbent functions for even $q$. We consider several primary and secondary constructions of such functions. It is proved that the numbers of self-dual and anti-self dual gbent functions coincide. We give necessary and sufficient conditions for the self-duality of Maiorana--McFarland gbent functions and find Hamming and Lee distances spectrums between them. We find all self-dual gbent functions symmetric with respect to two variables and prove that self-dual gbent function can not be affine. The properties of sign functions of self-dual gbent functions are considered. Symmetries that preserve self-duality are also discussed. 
\end{abstract}

\section{Introduction}

The study of Boolean functions having strong cryptographic properties is the domain of current interest, see monographies~\cite{Carlet_book,Cusick} for detail. Boolean bent functions were introduced by Rothaus~\cite{Rothaus} in $1976$. Due to maximal nonlinearity they have a number of applications in cryptography and coding theory. They were used as building blocks of stream (Grain, $2004$) and block (CAST, $1997$) ciphers and, for instance, in $2000$ T.~Wada~\cite{Wada} established a connection between bent functions and binary constant-amplitude codewords. But despite the long history of research in this area there are still many open problems. Among them the exact number of bent functions as well as their complete classification seems elusive to be solved for now. One can find more details on bent functions in books~\cite{Tokareva_book,Mesnager_book}. 

Bent functions were initially generalized by P.~V.~Kumar in $1985$ by considering functions of the form~$\mathbb{Z}_q^n\rightarrow\mathbb{Z}_q$ with corresponding definition of bentness, see~\cite{Kumar}. Bent functions from a finite Abelian group into a finite Abelian group were studied in~\cite{Solodovnikov} by~V.~I.~Solodovnikov. Having applications of functions from~$\mathbb{F}_2^n$ to~$\mathbb{Z}_4$ in code-division multiple access (CDMA) systems, K.-U~Schmidt in~\cite{Schmidt} (initially appeared in preprint from~$2006$) generalized the notion of bentness for functions of the form $\mathbb{F}_2^n\rightarrow\mathbb{Z}_q$, where $q\geqslant2$ is a positive integer and studied these functions for the case $q=4$. The considered functions are named generalized bent (gbent) functions. Note that this generalization deals with the mappings of the form $\mathbb{F}_2^n\rightarrow\mathbb{Z}_q$ called generalized Boolean functions, that are also studied from the view of obtaining linear codes with special properties, see~\cite{Paterson}. In recent years generalized bent functions obtained much attention, in particular, for the case $q=2^k$. In~\cite{HodzicConstructions,Martinsen,Hodzic2018,Generalized_bent} different constructions and properties of generalized bent functions were obtained. The connection between concepts of strong regularity of (edge-weighted) Cayley graph associated to a generalized Boolean function and gbent functions was pointed in~\cite{Riera}. The complete characterization of generalized bent functions from different perspectives was recently presented in~\cite{Hodzic,Tang,Mesnager2018}. A comprehensive survey on existing generalizations of bent functions can be found in~\cite{Tokareva_survey}.

For every bent Boolean function its dual bent functions is uniquely defined. It is important to note that the duality mapping is the unique known isometric mapping of the set of bent functions into itself that cannot be extended to a isometry of the whole set of all Boolean functions that preserves bentness. Self-dual bent functions form a remarkable class of bent functions since they have the direct relation to their dual bent functions and in terms of mappings can be considered as fixed points of the duality mapping. These functions were explored by~C.~Carlet et al. in $2010$ in work~\cite{Self-dual}, where a number of constructions and properties were given and the classification for small number of variables was provided. Next steps for the classification of qubic self-dual bent functions in $8$ variables were made in~\cite{Feulner}, while quadratic self-dual bent functions were completely characterized in~\cite{Hou}. Constructions and properties of self-dual Boolean bent functions were studied in~\cite{Hyun,Luo,Mesnager}. The overview of the known metrical properties of self-dual bent functions can be found in~\cite{Kutsenko_PDM}.

The action of the duality mapping on bent functions within generalizations is increasingly nontrivial since it is typically defined only for the part of bent functions from corresponding generalization which are called {\it regular}, while more accurate work with them also demands for intermediate notation like {\it weak regularity} that also appears in this scope. The extension of the concept of self-duality for different generalizations of bent functions was made in several papers. The classification of quadratic self-dual bent functions of the form $\mathbb{F}_p^n\rightarrow\mathbb{F}_p$, $p$ odd prime, was made by X.-D.~Hou in~\cite{Hou_generalized}. In paper~\cite{Cesmelioglu} the self-duality for bent functions within the same generalization type was studied by A.~\c{C}e\c{s}melio\u{g}lu et al. In~$2018$, L.~Sok. et al.~\cite{Quaternary} studied quaternary self-dual bent functions of the form $\mathbb{F}_2^n\rightarrow\mathbb{Z}_4$ from the viewpoints of existence, construction, and symmetry. The relation between sign functions of quaternary self-dual bent function in~$n$ variables and two self-dual bent functions in~$n$ variables was found. Based on this it was proved that there are no quaternary self-dual bent functions in odd number of variables.

In current work we investigate constructions, symmetries and other properties of self-dual generalized bent functions $\mathbb{F}_2^n\rightarrow\mathbb{Z}_q$, when~$q$ is even. The paper is organized as follows. The next section contains the necessary notation. In section~\ref{section:Constructions} several primary and secondary constructions are given. The metrical properties of self-dual gbent functions from Maiorana--McFarland class are characterized in section~\ref{section:Hamming and Lee distance spectrums}. In section~\ref{section:Sign functions of (anti-)self-dual gbent functions} sign functions of self-dual gbent functions are studied. Section~\ref{section:Properties of self-dual gbent function} deals with the properties of self-dual gbent functions including upper bound for the set of self-dual gbent functions, the existence of affine functions within self-dual gbent ones and characterization of self-dual gbent functions symmetric with respect to two variables. In Section~\ref{section:Symmetries} the construction of mappings preserving self-duality of gbent function is given.

\section{Notation}\label{section:Notation}

Let $\mathbb{F}_2^n$ be a set of binary vectors of length $n$. For $x,y\in\mathbb{F}_2^n$ denote $\langle x,y\rangle=\bigoplus\limits_{i=1}^n{x_iy_i}$, where the sign $\oplus$ denotes a sum modulo $2$. Denote, following~\cite{Janusz}, the orthogonal group of index $n$ over the field $\mathbb{F}_2$ as
\begin{equation*}
	\mathcal{O}_n=\left\{L\in \text{GL}\left(n,\mathbb{F}_2\right):LL^T=I_n\right\},
\end{equation*}
where $L^\mathrm{T}$ denotes the transpose of $L$ and $I_n$ is an identical matrix of order $n$ over the field~$\mathbb{F}_2$.

A {\it generalized Boolean function} $f$ in $n$ variables is any map from $\mathbb{F}_2^n$ to $\mathbb{Z}_q$, the integers modulo $q$. The set of generalized Boolean functions in $n$ variables is denoted by $\mathcal{GF}^q_n$, for the case $q=2$ we use $\mathcal{F}_n$. Let $\omega=e^{2\pi i/q}$. A {\it sign} function of $f\in\mathcal{GF}_n^q$ is a complex valued function $F=\omega^{f}$, we will also refer to it as to a complex vector $\left(\omega^{f_0},\omega^{f_1},...,\omega^{f_{2^n-1}}\right)$ of length $2^n$, where $\left(f_0,f_1,...,f_{2^n-1}\right)$ is a vector of values of the function $f$. 

The {\it Hamming weight} $\mathrm{wt}_H(x)$ of the vector $x\in\mathbb{F}_2^n$ is the number of nonzero coordinates of $x$. The {\it Hamming distance} $\mathrm{dist}_H(f,g)$ between generalized Boolean functions $f,g$ in $n$ variables is the cardinality of the set $\left\{x\in\mathbb{F}_2^n|f(x)\ne g(x)\right\}$. The Lee weight of the element $x\in\mathbb{Z}_q$ is~$\mathrm{wt}_{L}(x)=\min\left\{x,q-x\right\}$. The Lee distance~$\mathrm{dist}_L(f,g)$ between $f,g\in\mathcal{GF}^q_n$ is
\begin{equation*}
	\mathrm{dist}_{L}(f,g)=\sum\limits_{x\in\mathbb{F}_2^n}\mathrm{wt}_{L}\left(\delta(x)\right),
\end{equation*}
where $\delta\in\mathcal{GF}_n^q$ and $\delta(x)=f(x)+(q-1)g(x)$ for any $x\in\mathbb{F}_2^n$. For Boolean case $q=2$ the Hamming distance coincides with the Lee distance.

The {\it (generalized) Walsh--Hadamard transform} of $f\in\mathcal{GF}_n^q$ is the complex valued function:
\begin{equation*}
	H_f(y)=\sum\limits_{x\in\mathbb{F}_2^n}\omega^{f(x)}(-1)^{\langle x,y\rangle}.
\end{equation*}

A generalized Boolean function $f$ in $n$ variables is said to be {\it generalized bent} (gbent) if
\begin{equation*}
	\left\lvert H_f(y)\right\rvert=2^{n/2},
\end{equation*}
for all $y\in\mathbb{F}_2^{n}$~\cite{Schmidt}. If there exists such $\widetilde{f}\in\mathcal{GF}^q_n$ that $H_f(y)=\omega^{\widetilde{f}(y)}2^{n/2}$ for any~$y\in\mathbb{F}_2^n$, the gbent function $f$ is said to be {\it regular} and $\widetilde{f}$ is called its {\it dual}. Note that $\widetilde{f}$ is generalized bent as well. A regular gbent function $f$ is said to be {\it self-dual} if $f=\widetilde{f}$, and {\it anti-self-dual} if $f=\widetilde{f}+q/2$. Consequently, it is the case only for even $q$. So throughout this paper we assume that~$q$ is a positive even integer. Corresponding sets of gbent functions are denoted by $\mathrm{SB}_q^+(n)$ and $\mathrm{SB}_q^-(n)$, respectively.

\section{Constructions}\label{section:Constructions}

In this section we present several primary and seconady constructions of self-dual gbent functions.

\subsection{Direct sum}

Suppose $n=n_1+n_2+\ldots+n_r$, where $n_k$ are positive integers for~$k=1,2,\ldots,r$. Let $f\in\mathcal{GF}_n^q$, consider gbent functions $f_k\in\mathcal{GF}_{n_k}^{q}$, $k=1,2,\ldots,r$. The function
\begin{equation*}
	f(x)=f_1\big(x^{(1)}\big)+f_2\big(x^{(2)}\big)+\ldots+f_r\big(x^{(r)}\big),
\end{equation*}
where $x^{(k)}\in\mathbb{F}_2^{n_k}$ and $x=\big(x^{(1)},x^{(2)},\ldots,x^{(r)}\big)\in\mathbb{F}_2^n$, is a {\it direct sum} of generalized Boolean functions $f_k$. Gbent functions obtained by a direct sum of generalized Boolean functions were studied in paper~\cite{HodzicConstructions}, it was proved that the function $f$ is gbent if and only if all $f_k$ are gbent functions. Here we consider self-dual bent functions obtained by this construction.
\begin{proposition}
	Assume all numbers $p_k$ are even and $f_k\in\mathcal{GF}_{n_k}^{q}$ are gbent functions such that $\widetilde{f}_k=f_k+c_k\left(q/2\right)$, where $c_k\in\mathbb{F}_2$, $k=1,2,\ldots,r$. If there is an even number of nonzero coefficients $c_k$, then the function $f$ is a self-dual gbent function in $n$ variables.
\end{proposition}
\begin{proof}
	The Walsh--Hadamard transform of $f$ which is a direct sum of $f_k\in\mathcal{GF}_{n_k}^{q}$, $k=1,2,\ldots,r$, is given by
	\begin{align*}
		H_f(y)&=\sum\limits_{x\in\mathbb{F}_2^n}\omega^{f(x)}(-1)^{\langle x,y\rangle}=H_{f_1}\big(y^{(1)}\big)H_{f_2}\big(y^{(2)}\big)\cdot\ldots\cdot H_{f_r}\big(y^{(r)}\big)\\
		&=(-1)^{c_1+c_2+\ldots+c_r}2^{\big(n_1+n_2+\ldots+n_r\big)/2}\omega^{\widetilde{f}_1\big(y^{(1)}\big)+\widetilde{f}_2\big(y^{(2)}\big)+\ldots+\widetilde{f}_r\big(y^{(r)}\big)}\\
		&=(-1)^{c_1+c_2+\ldots+c_r}2^{n/2}\omega^{f_1\big(y^{(1)}\big)+f_2\big(y^{(2)}\big)+\ldots+f_r\big(y^{(r)}\big)}\\
		&=2^{n/2}\omega^{f(y)+(q/2)\big(c_1+c_2+\ldots+c_r\big)}
	\end{align*}
	for all $y^{(k)}\in\mathbb{F}_2^{n_k}$ and $y=\big(y^{(1)},y^{(2)},\ldots,y^{(r)}\big)\in\mathbb{F}_2^n$.
\end{proof}

\subsection{Maiorana--McFarland class}

Bent functions in $2k$ variables which have a representation
\begin{equation*}
	f(x,y)=\langle x,\pi(y)\rangle\oplus g(y),\ \ x,y\in\mathbb{F}_2^k,
\end{equation*}
where $\pi:\mathbb{F}_2^k\rightarrow\mathbb{F}_2^k$ is a permutation and $g$ is a Boolean function in $k$ variables, form the well known {\it Maiorana--McFarland} class of bent functions~\cite{McFarland}. It is known~\cite{Carlet} that the dual of a Maiorana--McFarland bent function $f(x,y)$ is equal to
\begin{equation*}
	\widetilde{f}(x,y)=\big\langle \pi^{-1}(x),y\big\rangle\oplus g\big(\pi^{-1}(x)\big), \ \ x,y\in\mathbb{F}_2^k.
\end{equation*}

A generalization of this construction for the case $q=4$ was given by Schmidt in~\cite{Schmidt}. In paper~\cite{Generalized_bent} this construction was given for any even $q$, thus, forming the following construction
\begin{equation*}
	f(x,y)=\frac{q}{2}\langle x,\pi(y)\rangle+g(y),\ \ x,y\in\mathbb{F}_2^k,
\end{equation*}
where $\pi:\mathbb{F}_2^k\rightarrow\mathbb{F}_2^k$ is a permutation and $g$ is a generalized Boolean function in $k$ variables. Its dual is
\begin{equation*}
	\widetilde{f}(x,y)=\frac{q}{2}\big\langle \pi^{-1}(x),y\big\rangle+g\big(\pi^{-1}(x)\big),\ \ x,y\in\mathbb{F}_2^k.
\end{equation*}
In the article~\cite{Self-dual} necessary and sufficient conditions of (anti-)self-duality of Maiorana--McFarland bent functions, denoted by $\mathrm{SB}_{\mathcal{M}}^+(n)$ ($\mathrm{SB}_{\mathcal{M}}^-(n)$), were given. In~\cite{Quaternary} quaternary self-dual Maiorana--McFarland bent functions were studied and necessary and sufficient conditions of self-duality were obtained for them. 

In the current work we generalize these results for any even $q$. Denote the sets of (anti-)self-dual generalized Maiorana--McFarland bent functions by $\mathrm{SB}_{\mathcal{M}^q}^+(n)$ ($\mathrm{SB}_{\mathcal{M}^q}^-(n)$ correspondingly).

\begin{theorem}\label{generalized_MM_self-duality}
	A generalized Maiorana--McFarland bent function
	\begin{equation*}
		f(x,y)=\frac{q}{2}\left\langle x,\pi(y)\right\rangle+g(y),\ x,y\in\mathbb{F}_2^{n/2},
	\end{equation*}
	is (anti-)self-dual bent if and only if for any $y\in\mathbb{F}_2^{n/2}$
	\begin{equation*}
		\pi(y)=L\left(y\oplus b\right),\ \ \
		g(y)=\frac{q}{2}\left\langle b,y\right\rangle+d,
	\end{equation*}
	where $L\in\mathcal{O}_{n/2},b\in\mathbb{F}_2^{n/2}$, $\mathrm{wt}\left(b\right)$ is even (odd), $d\in\mathbb{Z}_q$.
\end{theorem}

\begin{proof}
	Let $f(x,y)=\frac{q}{2}\left\langle x,\pi(y)\right\rangle\oplus g(y)$, where $\pi$ is a permutation on $\mathbb{F}_2^{n/2}$, $g\in\mathcal{GF}_{n/2}^q$, $x,y\in\mathbb{F}_2^{n/2}$. By the definition of (anti-)self-duality a generalized bent function is (anti-)self-dual if it coincides with (the complement of) its dual. Then for all $x,y\in\mathbb{F}_2^n$ it must hold
	\begin{equation}\label{MM function equals its dual}
		\frac{q}{2}\left\langle x,\pi(y)\right\rangle+g(y)=\frac{q}{2}\big\langle \pi^{-1}(x),y\big\rangle+g\big(\pi^{-1}(x)\big)+\frac{q}{2}c,
	\end{equation}
	where $c\in\mathbb{F}_2$: $c=0$ if $f=\widetilde{f}$ and $c=1$ if $f=\widetilde{f}+\frac{q}{2}$.
	
	Put zero vector $x=\bm{0}\in\mathbb{F}_2^{n/2}$ in~$(\ref{MM function equals its dual})$, then for any $y\in\mathbb{F}_2^n$ we have
	\begin{equation*}
		g(y)=\frac{q}{2}\big\langle \pi^{-1}\left(\bm{0}\right),y\big\rangle+g\big(\pi^{-1}\left(\bm{0}\big)\right)+\frac{q}{2}c.
	\end{equation*}
	
	The condition~$(\ref{MM function equals its dual})$ can be transformed to
	\begin{multline*}
		\frac{q}{2}\langle x,\pi(y)\rangle+\frac{q}{2}\big\langle \pi^{-1}\left(\bm{0}\right),y\big\rangle+g\big(\pi^{-1}\left(\bm{0}\right)\big)\\
		=\frac{q}{2}\big\langle \pi^{-1}(x),y\big\rangle+\frac{q}{2}\big\langle \pi^{-1}\left(\bm{0}\right),\pi^{-1}(x)\big\rangle+g\big(\pi^{-1}\left(\bm{0}\right)\big)+\frac{q}{2}c,
	\end{multline*}
	or, equivalently,
	\begin{equation}\label{MM function equals its dual1}
		\frac{q}{2}\langle x,\pi(y)\rangle+\frac{q}{2}\big\langle \pi^{-1}\left(\bm{0}\right),y\big\rangle=\frac{q}{2}\big\langle \pi^{-1}(x),y\big\rangle+\frac{q}{2}\big\langle \pi^{-1}\left(\bm{0}\right),\pi^{-1}(x)\big\rangle+\frac{q}{2}c.
	\end{equation}
	In both sides of~$(\ref{MM function equals its dual1})$ monomials of algebraic degree more than $2$ can not occur, since the left part has algebraic degree at most $1$ with respect to $x$ provided that $y$ is fixed and the right part has algebraic degree at most $1$ with respect to $y$ provided that $x$ is fixed. Therefore, the mapping $\pi$ is an affine permutation defined by $\pi(x)=L\left(x\oplus b\right)$ for any $x\in\mathbb{F}_2^n$, where $L$ is a $\left({n/2}\right)\times\left({n/2}\right)$ nonsingular binary matrix, $b\in\mathbb{F}_2^{n/2}$.
	
	Since the equality~$(\ref{MM function equals its dual1})$ should be considered by modulo $q$, we only care about the parity of components of both sides, thus, for any $x,y\in\mathbb{F}_2^{n/2}$ having the following equality
	\begin{equation}\label{MM function equals its dual2}
		\left\langle x,L\left(y\oplus b\right)\right\rangle\oplus\left\langle b,y\right\rangle=\big\langle L^{-1}x\oplus b,y\big\rangle\oplus\big\langle b,L^{-1}x\oplus b\big\rangle\oplus c.
	\end{equation}
	
	Putting $x\in\mathbb{F}_2^{n/2}$ to be zero vector in~$(\ref{MM function equals its dual2})$, then for any $y\in\mathbb{F}_2^{n/2}$ it must hold $\left\langle b,b\right\rangle=c$. Rewrite~$(\ref{MM function equals its dual2})$ in the form
	\begin{equation*}
		\big\langle x,Ly\oplus \big(L^{-1}\big)^Ty\big\rangle=\big\langle x,Lb\oplus\big(L^{-1}\big)^Tb\big\rangle,
	\end{equation*}
	and consider it for a zero vector $y=\bm{0}=\in\mathbb{F}_2^{n/2}$:
	\begin{equation*}
		\big\langle x,Lb\oplus\big(L^{-1}\big)^Tb\big\rangle=0,
	\end{equation*}
	that is $Lb\oplus\left(L^{-1}\right)^Tb=\bm{0}$ or, equivalently, $Lb=\left(L^{-1}\right)^Tb$. It means that 
	\begin{equation*}
		\big\langle x,\big(L\oplus \big(L^{-1}\big)^T\big)y\big\rangle=0,
	\end{equation*}
	for any $x,y\in\mathbb{F}_2^{n/2}$. From this it follows that $L^{-1}=L^T$, that is $L\in\mathcal{O}_{n/2}$.
\end{proof}

It follows that the number of such functions is a function of $q$ and the cardinality of the orthogonal group.

\begin{corollary}
	It holds
	\begin{equation*}
		\big\vert\mathrm{SB}^{+}_{\mathcal{GM}^q}(n)\big\vert=\big\vert\mathrm{SB}^{-}_{\mathcal{GM}^q}(n)\big\vert=q\cdot2^{n/2-1}\left\vert\mathcal{O}\left(n/2,\mathbb{F}_2\right)\right\vert.
	\end{equation*}
\end{corollary}

\subsection{Dillon functions type}

In paper~\cite{Martinsen} an explicit representation of functions in a generalization of Dillon’s~$\mathcal{PS}_{ap}$ class to gbent functions with $q=2^k$ was presented. By comparing the function from~$\mathcal{PS}_{ap}$ in a bivariate form with its dual we obtain the following result.
\begin{proposition}
	Assume $G_j$, $j=0,1,...,k-1$, be balanced Boolean functions in $m$ variables with $G_j(0)=0$ and $\sum\limits_{t\in\mathbb{F}_{2^m}}\omega^{\sum\limits_{j=0}^{k-1}2^jG_j(t)}=0$. Then, if $G_j(u)=G_j(1/u)$ for any $u\in\mathbb{F}_{2^m}$ (with the convention $1/0=0$), then the function $f:\mathbb{F}_{2^m}\times\mathbb{F}_{2^m}\rightarrow\mathbb{Z}_{2^k}$ given by
	\begin{equation}\label{equation:Dillon function}
		f(x,y)=\sum\limits_{j=0}^{k-1}2^jG_j(x/y),\ \ x,y\in\mathbb{F}_2^{n/2},
	\end{equation}
	is self-dual gbent in $2m$ variables.
\end{proposition}

\begin{proof}
	It is enough to mention that, as was shown in~\cite{Martinsen}, the dual gbent function of~(\ref{equation:Dillon function}) has the form
	\begin{equation*}
		\widetilde{f}(x,y)=\sum\limits_{j=0}^{k-1}2^jG_j(y/x),\ \ x,y\in\mathbb{F}_2^{n/2}.
	\end{equation*}
\end{proof}

\subsection{Iterative construction}

Let $f_0,f_1,f_2,f_3$ be Boolean functions in $n$ variables. Consider a Boolean function $g$ in $n+2$ variables which is defined as
\begin{equation*}
	g(00,x)=f_0(x),\ \ g(01,x)=f_1(x),
	\ \ g(10,x)=f_2(x),\ \ g(11,x)=f_3(x),\ \ 
	x\in\mathbb{F}_2^n.
\end{equation*}
It is known (Preneel et al.~\cite{Preneel}; see also~\cite{Tokareva_iterative}) that under condition that all functions~$f_0,f_1,f_2,f_3$ are Boolean bent functions in $n$ variables, the mentioned function $g$ is a bent function in $n+2$ variables if and only if 
\begin{equation*}
	\widetilde{f_0}\oplus\widetilde{f_1}\oplus\widetilde{f_2}\oplus\widetilde{f_3}=1,
\end{equation*} 
that gives the construction of a bent function in $n+2$ variables through the concatenation of vectors of values of four bent functions in $n$ variables~\cite{Preneel}.

Following N.\;Tokareva~\cite{Tokareva_iterative}, we will refer to Boolean bent functions obtained by this construction as {\it bent iterative functions} $\left(\mathcal{BI}\right)$. A construction of generalized bent functions in~$n+2$ variables obtained by a concatenation of four generalized Boolean functions on $n$ variables was studied in~\cite{Singh}.

Bent iterative constructions of self-dual Boolean bent functions in $n+2$ variables, based on concatenation of $4$ Boolean bent functions in $n$ variables, were presented in~\cite{Self-dual,Kutsenko_DCC}. In current work we give two constructions of generalized bent iterative functions that generalize the constructions for Boolean case:
\begin{proposition}\label{proposition:Iterative}
	\begin{itemize}
		\item[1)] Let $f$ be a regular gbent function in $n$ variables, then the sign function
		\begin{equation*}
			\big(F,\widetilde{F},\widetilde{F},-F\big),
		\end{equation*}
		where $F=\omega^f$ and $\widetilde{F}=\omega^{\widetilde{f}}$, is the sign function of a self-dual gbent function in $n+2$ variables;
		\item[2)] Let $f$ be a self-dual gbent function in $n$ variables with the sign function $F$, and $g$ be an anti-self-dual gbent function in $n$ variables with the sign function $G$, then the sign function
		\begin{equation*}
			\big(F,G,-G,F\big),
		\end{equation*}
		where $F=\omega^f$ and $G=\omega^g$, is the sign function of a gbent function in $n+2$ variables.
	\end{itemize}
\end{proposition}

	\begin{proof}
		Let $F=\omega^f$ be a sign function of regular gbent function $f$ in~$n$ variables. It is clear that the function $h$ is self-dual gbent if and only if
		\begin{align*}
			\mathcal{H}_{n+2}\begin{pmatrix}
				F\\
				\widetilde{F}\\
				\widetilde{F}\\
				-F
			\end{pmatrix}&=
			\frac{1}{2}\begin{pmatrix}
				\mathcal{H}_n& \mathcal{H}_n& \mathcal{H}_n& \mathcal{H}_n\\
				\mathcal{H}_n& -\mathcal{H}_n& \mathcal{H}_n& -\mathcal{H}_n\\
				\mathcal{H}_n& \mathcal{H}_n & -\mathcal{H}_n& -\mathcal{H}_n\\
				\mathcal{H}_n& -\mathcal{H}_n& -\mathcal{H}_n& \mathcal{H}_n
			\end{pmatrix}
			\begin{pmatrix}
				F\\
				\widetilde{F}\\
				\widetilde{F}\\
				-F
			\end{pmatrix}\\
		&=\frac{1}{2}
			\begin{pmatrix}
				\widetilde{F}+F+F-\widetilde{F}\\
				\widetilde{F}-F+F+\widetilde{F}\\
				\widetilde{F}+F-F+\widetilde{F}\\
				\widetilde{F}-F-F-\widetilde{F}
			\end{pmatrix}
		=\begin{pmatrix}
			F\\
			\widetilde{F}\\
			\widetilde{F}\\
			-F
		\end{pmatrix},
		\end{align*}
		Let $f$ be a self-dual gbent function in $n$ variables with the sign function $F=\omega^f$, and $g$ be an anti-self-dual gbent function in $n$ variables with the sign function $G=\omega^g$, then
		\begin{align*}
			\mathcal{H}_{n+2}\begin{pmatrix}
				F\\
				G\\
				-G\\
				F
			\end{pmatrix}
			&=
			\frac{1}{2}\begin{pmatrix}
				\mathcal{H}_n& \mathcal{H}_n& \mathcal{H}_n& \mathcal{H}_n\\
				\mathcal{H}_n& -\mathcal{H}_n& \mathcal{H}_n& -\mathcal{H}_n\\
				\mathcal{H}_n& \mathcal{H}_n & -\mathcal{H}_n& -\mathcal{H}_n\\
				\mathcal{H}_n& -\mathcal{H}_n& -\mathcal{H}_n& \mathcal{H}_n
			\end{pmatrix}
			\begin{pmatrix}
				F\\
				G\\
				-G\\
				F
			\end{pmatrix}\\
			&=\frac{1}{2}
			\begin{pmatrix}
				F+G-G+F\\
				F-G-G-F\\
				F+G+G-F\\
				F-G+G+F
			\end{pmatrix}
			=\begin{pmatrix}
				-F\\
				-G\\
				G\\
				-F
			\end{pmatrix},
		\end{align*}
	\end{proof}

\section{Hamming and Lee distance spectrums}\label{section:Hamming and Lee distance spectrums}

The spectrum of Hamming distances between self-dual Maiorana--McFraland Boolean bent functions was studied in~\cite{Kutsenko_MM}. It was proved that
\begin{equation*}
	\mathrm{Sp}_H\left(\mathrm{SB}^{+}_{\mathcal{M}}(n)\cup\mathrm{SB}^{-}_{\mathcal{M}}(n)\right)=\left\{2^{n-1}\right\}\cup\bigcup\limits_{r=0}^{n/2-1}\left\{2^{n-1}\left(1\pm\frac{1}{2^r}\right)\right\},
\end{equation*}
and, if either $f,g\in\mathrm{SB}_{\mathcal{M}}^+(n)$ or $f,g\in\mathrm{SB}_{\mathcal{M}}^-(n)$, then all distances except $2^{n-1}$ are attainable, and for any pair $f\in\mathrm{SB}_{\mathcal{M}}^+(n)$ and $g\in\mathrm{SB}_{\mathcal{M}}^-(n)$ it holds $\mathrm{dist}(f,g)=2^{n-1}$.

\subsection{Hamming distance spectrum}

For generalized case we have
\begin{proposition}
	It holds
	\begin{equation*}
		\mathrm{Sp}_H\left(\mathrm{SB}^{+}_{\mathcal{GM}^q}(n)\cup\mathrm{SB}^{-}_{\mathcal{GM}^q}(n)\right)=\mathrm{Sp}_H\left(\mathrm{SB}^{+}_{\mathcal{M}}(n)\cup\mathrm{SB}^{-}_{\mathcal{M}}(n)\right).
	\end{equation*}
	Moreover, all given distances are attainable.
\end{proposition}
\begin{proof}
	Let $f_1,f_2\in\mathrm{SB}^{+}_{\mathcal{GM}^q}(n)\cup\mathrm{SB}^{-}_{\mathcal{GM}^q}(n)$. We have
	\begin{equation*}
		f_1(x,y)=\frac{q}{2}h_1(x,y)+d_1,\ \ x,y\in\mathbb{F}_2^{n/2},
	\end{equation*}
	\begin{equation*}
		f_2(x,y)=\frac{q}{2}h_2(x,y)+d_2,\ \ x,y\in\mathbb{F}_2^{n/2},
	\end{equation*}
	for some $h_1,h_2\in\mathrm{SB}^+_{\mathcal{M}}(n)\cup\mathrm{SB}^-_{\mathcal{M}}(n)$ and $d_1,d_2\in\mathbb{Z}_q$. If $\mathrm{wt}\left(d_1-d_2\right)\notin\left\{0,q/2\right\}$, then $\mathrm{dist}_L\left(f_1,f_2\right)=2^n$. Otherwise, the distance coincides with some value from the spectrum for binary case so by taking $d_1,d_2=0$ and varying $h_1,h_2$ this spectrum can be entirely covered.
\end{proof}

\subsection{Lee distance spectrum}

For binary case the Hamming distance coincides with the Lee distance, so for this case the Lee distance spectrum follows. For $q>2$ the spectrum can be obtained by using the set of attainable Hamming distances from binary case.	
\begin{theorem}
	It holds
	\begin{multline*}
		\mathrm{Sp}_L\left(\mathrm{SB}^{+}_{\mathcal{GM}^q}(n)\cup\mathrm{SB}^{-}_{\mathcal{GM}^q}(n)\right)\\
		=\left\{q\cdot2^{n-2}\right\}\cup\bigcup\limits_{w=0}^{q/2}\bigcup\limits_{r=0}^{n/2-1}\left\{q\cdot 2^{n-2}\left(1\pm\frac{1}{2^r}\right)\mp w\cdot 2^{n-r}\right\}.
	\end{multline*}
	Moreover, all given distances are attainable.
\end{theorem}
\begin{proof}
	Let $f_1,f_2\in\mathrm{SB}^{+}_{\mathcal{GM}^q}(n)\cup\mathrm{SB}^{-}_{\mathcal{GM}^q}(n)$. Again, we have
	\begin{equation*}
		f_1(x,y)=\frac{q}{2}h_1(x,y)+d_1,\ \ x,y\in\mathbb{F}_2^{n/2},
	\end{equation*}
	\begin{equation*}
		f_2(x,y)=\frac{q}{2}h_2(x,y)+d_2,\ \ x,y\in\mathbb{F}_2^{n/2},
	\end{equation*}
	for some $h_1,h_2\in\mathrm{SB}^+_{\mathcal{M}}(n)\cup\mathrm{SB}^-_{\mathcal{M}}(n)$ and $d_1,d_2\in\mathbb{Z}_q$. Denote $\mathrm{wt}_L\left(d_1-d_2\right)$ by $w$ and the Hamming distance $\mathrm{dist}_H\left(h_1,h_2\right)$ between Boolean functions $h_1,h_2$ by $d$. Under this notation the Lee distance between $f_1$ and $f_2$ is the function of $w$, $d$ and number of variables $n$. Indeed,
	\begin{align*}
		\mathrm{dist}_L\left(f_1,f_2\right)&=\sum\limits_{x,y\in\mathbb{F}_2^{n/2}}\mathrm{wt}_L\left(f_1(x,y)-f_2(x,y)\right)\\
		&=\sum\limits_{\substack{x,y\in\mathbb{F}_2^{n/2}\\h_1(x,y)=h_2(x,y)}}\mathrm{wt}_L\left(f_1(x,y)-f_2(x,y)\right)\\
		&+\sum\limits_{\substack{x,y\in\mathbb{F}_2^{n/2}\\h_1(x,y)\ne h_2(x,y)}}\mathrm{wt}_L\left(f_1(x,y)-f_2(x,y)\right)\\
		&=\sum\limits_{\substack{x,y\in\mathbb{F}_2^{n/2}\\h_1(x,y)=h_2(x,y)}}\mathrm{wt}_L\left(d_1-d_2\right)\\
		&+\sum\limits_{\substack{x,y\in\mathbb{F}_2^{n/2}\\h_1(x,y)\ne h_2(x,y)}}\mathrm{wt}_L\left(d_1-d_2+\frac{q}{2}\right)\\
		&=\left(2^n-d\right)w+d\left(\frac{q}{2}-w\right)=2^nw-2dw+\frac{q}{2}d.
	\end{align*}
	
	If $h_1\in\mathrm{SB}^+_{\mathcal{M}}(n)$ and $h_2\in\mathrm{SB}^-_{\mathcal{M}}(n)$, then
	\begin{equation*}
		\mathrm{dist}_L\left(f_1,f_2\right)=2^nw-2\cdot2^{n-1}w+\frac{q}{2}\cdot2^{n-1}=q\cdot 2^{n-2}.
	\end{equation*}
	
	From the aforementioned Hamming distance spectrum for binary case it follows that for any $r\in\{0,1,...,n/2-1\}$ there exists at least one pair of (anti-)self-dual Boolean Maiorana--McFarland bent functions in $n$ variables at the Hamming distance $d=2^{n-1}+2^{n-r-1}$ as well as $2^n-d=2^{n-1}-2^{n-r-1}$. Assume $r$ is fixed, put $d=2^{n-1}\left(1\pm 2^{-r}\right)$ in the expression for $\mathrm{dist}_L\left(f_1,f_2\right)$:
	\begin{align*}
		\mathrm{dist}_L\left(f_1,f_2\right)&=2^nw-2w\cdot 2^{n-1}\left(1\pm\frac{1}{2^r}\right)+\frac{q}{2}\cdot 2^{n-1}\left(1\pm\frac{1}{2^r}\right)\\
		&=2^nw-2^nw\mp w\cdot 2^{n-r}+q\cdot 2^{n-2}\left(1\pm\frac{1}{2^r}\right)\\
		&=q\cdot2^{n-2}\left(1\pm\frac{1}{2^r}\right)\mp w\cdot 2^{n-r}.
	\end{align*}
	
	Observation that $r$ runs $\{0,1,...,n/2-1\}$ and $w$ varies within the set $\{0,1,...,q/2\}$ yields the result.		
\end{proof}

\begin{proposition}
	The minimal Lee distance between generalized (anti-)self-dual Maiorana--McFarland bent functions in $n$ variables is equal to $q\cdot 2^{n-3}$.
\end{proposition}
\begin{proof}
	Estimate the minimal value of the term
	\begin{equation*}
		D=q\cdot2^{n-2}\left(1\pm\frac{1}{2^r}\right)\mp w\cdot 2^{n-r},
	\end{equation*}
	with $r\in\{1,2,...,n/2-1\}$ and $w\in\{0,1,...,q/2\}$. Here we exclude the case $r=0$ since then the Lee distance is equal to either $D=w\cdot 2^n\geqslant 2^n$ or $D=2^{n-1}\left(q-2w\right)\geqslant 2^n$, provided that $f_1,f_2$ are distinct. Indeed, $r=0$ implies $d\in\left\{0,2^n\right\}$, and the first aforementioned expression for $D$ corresponds to $h_1=h_2$, while the second one to $h_1\oplus h_2=1$.
	
	Consider two cases depending on sequence of the signs. 
	
	\underline{Case 1:} 
	\begin{equation*}
		D=q\cdot2^{n-2}\left(1+\frac{1}{2^r}\right)-w\cdot 2^{n-r}=q\cdot 2^{n-2}+2^{n-r}\left(\frac{q}{4}-w\right).
	\end{equation*}
	
	Since $w\in\{0,1,...,q/2\}$ it follows that
	\begin{equation*}
		-\frac{q}{4}\leqslant\frac{q}{4}-w\leqslant\frac{q}{4},
	\end{equation*}
	hence
	\begin{equation*}
		-\frac{q}{4}\cdot 2^{n-r}\leqslant 2^{n-r}\left(\frac{q}{4}-w\right)\leqslant\frac{q}{4}\cdot 2^{n-r}.
	\end{equation*}
	
	Then
	\begin{equation*}
		D\geqslant q\cdot2^{n-2}-\frac{q}{4}\cdot 2^{n-r}=q\cdot 2^{n-2}\left(1-\frac{1}{2^r}\right)\geqslant q\cdot 2^{n-3},
	\end{equation*}
	and $D_{\min}=q\cdot 2^{n-3}$, that is attainable for $r=1$ and $w=q/2$.
	
	\underline{Case 2:} 
	\begin{equation*}
		D=q\cdot2^{n-2}\left(1-\frac{1}{2^r}\right)+w\cdot 2^{n-r}=q\cdot 2^{n-2}+2^{n-r}\left(w-\frac{q}{4}\right).
	\end{equation*}
	
	From $w\in\{0,1,...,q/2\}$ it follows that
	\begin{equation*}
		-\frac{q}{4}\leqslant w-\frac{q}{4}\leqslant\frac{q}{4},
	\end{equation*}
	hence
	\begin{equation*}
		-\frac{q}{4}\cdot 2^{n-r}\leqslant 2^{n-r}\left(w-\frac{q}{4}\right)\leqslant\frac{q}{4}\cdot 2^{n-r}.
	\end{equation*}
	
	Then
	\begin{equation*}
		D\geqslant q\cdot2^{n-2}-\frac{q}{4}\cdot 2^{n-r}=q\cdot 2^{n-2}\left(1-\frac{1}{2^r}\right)\geqslant q\cdot 2^{n-3},
	\end{equation*}
	and again $D_{\min}=q\cdot 2^{n-3}$, that is attainable for $r=1$ and $w=0$.
	
	Thus the minimal Lee distance is equal to $q\cdot 2^{n-3}$.
\end{proof}

In~\cite{Paterson} it was shown that both minimal Hamming and Lee distances of generalized Reed--Muller codes $\mathrm{RM}_q(r,n)$ are equal to~$2^{n-r}$ for any positive integer~$q$. Therefore, it immediately follows that
\begin{corollary}
	The minimal Hamming distance $2^{n-2}$ between quadratic (generalized) bent functions is attainable on the sets of self-dual and anti-self-dual  Maiorana--McFarland bent functions from $\mathcal{GM}^q_n$ only for $q=2$.
\end{corollary}	

\section{Sign functions of (anti-)self-dual gbent functions}\label{section:Sign functions of (anti-)self-dual gbent functions}

Let $I_n$ be the identity matrix of size $n$ and $H_n=H_1^{\otimes n}$ be the $n$-fold tensor product of the matrix $H_1$ with itself, where
\begin{equation*}
	H_1=\begin{pmatrix}
		1 & 1 \\
		1 & -1
	\end{pmatrix}.
\end{equation*}
It is known the Hadamard property of this matrix
\begin{equation*}
	H_nH_n^T=2^nI_{2^n},
\end{equation*}
where $H_n^{T}$ is transpose of $H_n$ (it holds $H_n^T=H_n$ by symmetricity of $H_n$). Denote~$\mathcal{H}_n=2^{-n/2}H_n$.

Recall an orthogonal decomposition of $\mathbb{R}^{2^n}$ in eigenspaces of $H_n$ from~\cite{Self-dual} (Lemma~5.2):
\begin{equation*}
	\mathbb{R}^{2^n}=\mathrm{Ker}\left(H_n+2^{n/2}I_{2^n}\right)\oplus\mathrm{Ker}\left(H_n-2^{n/2}I_{2^n}\right),
\end{equation*}
where the symbol $\oplus$ denotes a direct sum of subspaces. Consider the same decomposition
\begin{equation*}
	\mathbb{C}^{2^n}=\mathrm{Ker}\left(H_n+2^{n/2}I_{2^n}\right)\oplus\mathrm{Ker}\left(H_n-2^{n/2}I_{2^n}\right),
\end{equation*}
for a complex space $\mathbb{C}^{2^n}$.

As for the Boolean case (see~\cite{Kutsenko_CC}), we note that sign function of any self-dual gbent function is the eigenvector of $\mathcal{H}_n$ attached to the eigenvalue~$(+1)$, that is an element from the subspace $\mathrm{Ker}\big(\mathcal{H}_n-I_{2^n}\big)=\mathrm{Ker}\big(H_n-2^{n/2}I_{2^n}\big)$. The same holds for a sign function of any anti-self-dual gbent function, which obviously is an eigenvector of $\mathcal{H}_n$ attached to the eigenvalue~$(-1)$, that is an element from the subspace $\mathrm{Ker}\big(\mathcal{H}_n+I_{2^n}\big)=\mathrm{Ker}\big(H_n+2^{n/2}I_{2^n}\big)$. 

It is known that
\begin{equation*}
	\mathrm{dim}\big(\mathrm{Ker}\left(\mathcal{H}_n+I_{2^n}\right)\big)=\mathrm{dim}\big(\mathrm{Ker}\big(\mathcal{H}_n-I_{2^n}\big)\big)=2^{n-1}, 
\end{equation*}
where $\mathrm{dim}(V)$ is the dimension of the subspace $V\subseteq\mathbb{R}^{2^n}$. Moreover, since $\mathcal{H}_n$ is symmetric (Hermitian), the subspaces $\mathrm{Ker}\big(\mathcal{H}_n+I_{2^n}\big)$ and $\mathrm{Ker}\big(\mathcal{H}_n-I_{2^n}\big)$ are mutually orthogonal.

In~\cite{Kutsenko_DCC} it was proved that provided $n\geqslant4$, the linear span of sign functions of self-dual as well as anti-self-dual Boolean bent functions Boolean bent functions in $n$ variables has dimension $2^{n-1}$. The same result can be also given for gbent functions:
\begin{theorem}\label{basis_generalized}
	Let $n\geqslant4$, then the linear span of sign functions of (anti-)self-dual gbent functions in $n$ variables has dimension $2^{n-1}$.
\end{theorem}

\begin{proof}
	It is enough to mention that since $q$ is even it holds $(-1)=\omega^{q/2}\in\big\{\omega,\omega^2,...,\omega^{q-1}\big\}$, therefore the set of sign fuctions of (anti-)self-dual Boolean bent functions in $n$ variables is a subset of the set of sign functions of (anti-)self-dual gbent functions in $n$ variables. Then from~\cite{Kutsenko_DCC} (Theorem~2) the dimension follows.
\end{proof}

It is worth to note that the example of the basis of the subspace $\mathrm{Ker}\big(\mathcal{H}_n-I_{2^n}\big)$ can be constructed by using the functions obtained from the construction from Proposition~\ref{proposition:Iterative}.

When $n=2$ there are two self-dual Boolean bent functions, namely $x_1x_2$ and $x_1x_2\oplus1$, which have sign functions $(1,1,1,-1)$ and $(-1,-1,-1,1)$ respectively. These sign functions are linearly dependent vectors in $\mathbb{R}^4$. The set $\mathrm{SB}^-(2)$ consists of functions $x_1x_2\oplus x_1\oplus x_2$ and $x_1x_2\oplus x_1\oplus x_2\oplus 1$ with sign functions $(1,-1,-1,-1)$ and $(-1,1,1,1)$ respectively. These sign functions are linearly dependent vectors in $\mathbb{R}^4$ as well. Generalization comprises solution of the system
\begin{equation*}
	\frac{1}{2}\begin{pmatrix}
		1& 1& 1& 1\\
		1& -1& 1& -1\\
		1& 1 & -1& -1\\
		1& -1& -1& 1
	\end{pmatrix}
	\begin{pmatrix}
		\omega^{d_1}\\
		\omega^{d_2}\\
		\omega^{d_3}\\
		\omega^{d_4}
	\end{pmatrix}=
	\begin{pmatrix}
		\omega^{d_1}\\
		\omega^{d_2}\\
		\omega^{d_3}\\
		\omega^{d_4}
	\end{pmatrix},
\end{equation*}
where variables are numbers $d_1,d_2,d_3,d_4\in\mathbb{Z}_q$ in fact. It is clear that the only solution pattern is
\begin{equation*}
	\big(\omega^d,\omega^d,\omega^d,\omega^{d+q/2}\big)=\omega^d\cdot\left(1,1,1,-1\right)\in\mathbb{C}^4,
\end{equation*}
where $d\in\mathbb{Z}_q$. It means that any two sign functions of self-dual gbent functions from $\mathrm{SB}^+_q(2)$ are linearly dependent over $\mathbb{C}$ and $\big\lvert\mathrm{SB}^+_q(2)\big\rvert=q$.

The next result is a generalization of the similar one from~\cite{Kutsenko_DCC}.
\begin{theorem}\label{theorem:components}
	Let $n\geqslant4$ and $f\in\mathrm{SB}_q^+(n)$. For sign function $\omega^f=\big(F^{00},F^{01},F^{10},F^{11}\big)$, where $F^{00},F^{01},F^{10},F^{11}\in\big\{1,\omega,\omega^2,\ldots,\omega^{q-1}\big\}^{2^{n-2}}$, it holds
	\begin{align*}
		\big\langle F^{00},F^{01}\big\rangle+\big\langle F^{10},F^{11}\big\rangle&=0,\\
		\big\langle F^{00},F^{10}\big\rangle+\big\langle F^{01},F^{11}\big\rangle&=0.
	\end{align*}
\end{theorem}

\begin{proof}
	Let~$f\in\mathrm{SB}_q^+(n)$, then by Theorem~\ref{basis_generalized} there exist vectors
	\begin{align*}
		\alpha&=\left(\alpha_1,\alpha_2,\ldots,\alpha_{2^{n-3}}\right)\in\mathbb{C}^{2^{n-3}},\\
		\beta&=\left(\beta_1,\beta_2,\ldots,\beta_{2^{n-3}}\right)\in\mathbb{C}^{2^{n-3}},\\
		\gamma&=\left(\gamma_1,\gamma_2,\ldots,\gamma_{2^{n-2}}\right)\in\mathbb{C}^{2^{n-2}},
	\end{align*}
	such that
	\begin{equation*}
		\omega^f=\sum\limits_{i=1}^{2^{n-3}}{\alpha_i{\mathbf{F}}_i^n}+\sum\limits_{j=1}^{2^{n-3}}{\beta_j{\mathbf{G}}_j^n}+\sum\limits_{k=1}^{2^{n-2}}{\gamma_k\left({\mathbf{FG}}\right)_k^n},
	\end{equation*}
	where the sets $S_{\mathbf{F}}=\left\{\mathbf{F}_i^n\right\}_{i=1}^{2^{n-3}}$, $S_{\mathbf{G}}=\left\{\mathbf{G}_j^n\right\}_{j=1}^{2^{n-3}}$ and $S_{\mathbf{FG}}=\left\{\left(\mathbf{FG}\right)_k^n\right\}_{k=1}^{2^{n-2}}$ are described in the proof of Theorem~$2$ from~\cite{Kutsenko_DCC}. Consider the sets $S_{\mathbf{F}}$, $S_{\mathbf{G}}$, $S_{\mathbf{FG}}$ and denote 
	\begin{align*}
		\mathbf{F}_i^n&=\left(F_i,F_i,F_i,-F_i\right),\\
		\mathbf{G}_j^n&=\left(G_j,-G_j,-G_j,-G_j\right),\\ \left(\mathbf{FG}\right)_k^n&=\left(A_k,-B_k,B_k,A_k\right), 
	\end{align*}
	where $F_i,A_k\in\mathrm{Ker}\big(\mathcal{H}_{n-2}-I_{2^{n-2}}\big)$, $G_j,B_k\in\mathrm{Ker}\big(\mathcal{H}_{n-2}+I_{2^{n-2}}\big)$, $i,j=1,2,\ldots,2^{n-3}$, $k=1,2,\ldots,2^{n-2}$, and define the vectors
	\begin{align*}
		\mathbf{F}&=\sum\limits_{i=1}^{2^{n-3}}\alpha_iF_i,&
		\mathbf{G}&=\sum\limits_{j=1}^{2^{n-3}}\beta_jG_j,& 
		\mathbf{A}&=\sum\limits_{k=1}^{2^{n-2}}\gamma_kA_k,&
		\mathbf{B}&=\sum\limits_{k=1}^{2^{n-2}}\gamma_kB_k.
	\end{align*}
	Under this notation the sign function $\omega^f$ has the form
	\begin{equation*}
		\omega^f=
		\begin{pmatrix}
			F^{00}\\
			F^{01}\\
			F^{10}\\
			F^{11}
		\end{pmatrix}
		=
		\begin{pmatrix}
			\mathbf{F}+\mathbf{G}+\mathbf{A}\\
			\mathbf{F}-\mathbf{G}-\mathbf{B}\\
			\mathbf{F}-\mathbf{G}+\mathbf{B}\\
			-\mathbf{F}-\mathbf{G}+\mathbf{A}
		\end{pmatrix}\in\big\{1,\omega,\omega^2,...,\omega^{q-1}\big\}^{2^n}.
	\end{equation*}
	
	For any $j=1,2,\ldots,2^{n-2}$ denote
	\begin{align*}
		&\big(\mathbf{F}+\mathbf{G}\big)_j+\mathbf{A}_j=\omega^{t_j},\\
		&\big(\mathbf{F}-\mathbf{G}\big)_j-\mathbf{B}_j=\omega^{r_j},\\
		&\big(\mathbf{F}-\mathbf{G}\big)_j+\mathbf{B}_j=\omega^{l_j},\\
		&-\big(\mathbf{F}+\mathbf{G}\big)_j+\mathbf{A}_j=\omega^{k_j},
	\end{align*}
	where $t_j,r_j,l_j,k_j\in\mathbb{Z}_q$. Then
	\begin{align*}
		\mathbf{A}_j&=\frac{1}{2}\big(\omega^{t_j}+\omega^{k_j}\big),\\
		\mathbf{B}_j&=\frac{1}{2}\big(\omega^{l_j}-\omega^{r_j}\big),\\
		\big(\mathbf{F}+\mathbf{G}\big)_j&=\frac{1}{2}\big(\omega^{t_j}-\omega^{k_j}\big),\\
		\big(\mathbf{F}-\mathbf{G}\big)_j&=\frac{1}{2}\big(\omega^{r_j}+\omega^{l_j}\big).
	\end{align*}

	Note that
	\begin{equation*}
		\scal{\mathbf{G}}{\mathbf{A}}=\scal{\mathbf{F}}{\mathbf{B}}=0.
	\end{equation*}
	By using this we obtain the expression for the first inner product
	\begin{align}\label{equation:the first inner product}
		\big\langle F^{00},F^{01}\big\rangle+\big\langle F^{10},F^{11}\big\rangle&=\scal{\mathbf{F}+\mathbf{G}+\mathbf{A}}{\mathbf{F}-\mathbf{G}-\mathbf{B}}\nonumber\\
		&+\scal{\mathbf{F}-\mathbf{G}+\mathbf{B}}{-\mathbf{F}-\mathbf{G}+\mathbf{A}}\nonumber\\
		&=\scal{\mathbf{F}}{\mathbf{F}}-\scal{\mathbf{F}}{\mathbf{G}}-\scal{\mathbf{F}}{\mathbf{B}}\nonumber\\
		&+\scal{\mathbf{G}}{\mathbf{F}}-\scal{\mathbf{G}}{\mathbf{G}}-\scal{\mathbf{G}}{\mathbf{B}}\nonumber\\
		&+\scal{\mathbf{A}}{\mathbf{F}}-\scal{\mathbf{A}}{\mathbf{G}}-\scal{\mathbf{A}}{\mathbf{B}}\nonumber\\
		&-\scal{\mathbf{F}}{\mathbf{F}}-\scal{\mathbf{F}}{\mathbf{G}}+\scal{\mathbf{F}}{\mathbf{A}}\nonumber\\
		&+\scal{\mathbf{G}}{\mathbf{F}}+\scal{\mathbf{G}}{\mathbf{G}}-\scal{\mathbf{G}}{\mathbf{A}}\nonumber\\
		&-\scal{\mathbf{B}}{\mathbf{F}}-\scal{\mathbf{B}}{\mathbf{G}}+\scal{\mathbf{B}}{\mathbf{A}}\nonumber\\
		&=\scal{\mathbf{A}}{\mathbf{F}}+\scal{\mathbf{F}}{\mathbf{A}}-\scal{\mathbf{G}}{\mathbf{B}}-\scal{\mathbf{B}}{\mathbf{G}} \nonumber\\
		&=\scal{\mathbf{A}}{\mathbf{F}+\mathbf{G}}+\overline{\scal{\mathbf{A}}{\mathbf{F}+\mathbf{G}}}
		+\scal{\mathbf{B}}{\mathbf{F}-\mathbf{G}}+\overline{\scal{\mathbf{B}}{\mathbf{F}-\mathbf{G}}}
	\end{align}
	while the second one has the form
	\begin{align}\label{equation:the second inner product}
		\big\langle F^{00},F^{10}\big\rangle+\big\langle F^{01},F^{11}\big\rangle&=\scal{\mathbf{F}+\mathbf{G}+\mathbf{A}}{\mathbf{F}-\mathbf{G}+\mathbf{B}}\nonumber\\
		&+\scal{\mathbf{F}-\mathbf{G}-\mathbf{B}}{-\mathbf{F}-\mathbf{G}+\mathbf{A}}\nonumber\\
		&=\scal{\mathbf{F}}{\mathbf{F}}-\scal{\mathbf{F}}{\mathbf{G}}+\scal{\mathbf{F}}{\mathbf{B}}\nonumber\\
		&+\scal{\mathbf{G}}{\mathbf{F}}-\scal{\mathbf{G}}{\mathbf{G}}+\scal{\mathbf{G}}{\mathbf{B}}\nonumber\\
		&+\scal{\mathbf{A}}{\mathbf{F}}-\scal{\mathbf{A}}{\mathbf{G}}+\scal{\mathbf{A}}{\mathbf{B}}\nonumber\\
		&-\scal{\mathbf{F}}{\mathbf{F}}-\scal{\mathbf{F}}{\mathbf{G}}+\scal{\mathbf{F}}{\mathbf{A}}\nonumber\\
		&+\scal{\mathbf{G}}{\mathbf{F}}+\scal{\mathbf{G}}{\mathbf{G}}-\scal{\mathbf{G}}{\mathbf{A}}\nonumber\\
		&+\scal{\mathbf{B}}{\mathbf{F}}+\scal{\mathbf{B}}{\mathbf{G}}-\scal{\mathbf{B}}{\mathbf{A}}\nonumber\\
		&=\scal{\mathbf{A}}{\mathbf{F}}+\scal{\mathbf{F}}{\mathbf{A}}+\scal{\mathbf{G}}{\mathbf{B}}+\scal{\mathbf{B}}{\mathbf{G}} \nonumber\\
		&=\scal{\mathbf{A}}{\mathbf{F}+\mathbf{G}}+\overline{\scal{\mathbf{A}}{\mathbf{F}+\mathbf{G}}}
		-\scal{\mathbf{B}}{\mathbf{F}-\mathbf{G}}-\overline{\scal{\mathbf{B}}{\mathbf{F}-\mathbf{G}}}
	\end{align}
	Consider inner in details the following inner products
	\begin{align*}
		\scal{\mathbf{A}}{\mathbf{F}+\mathbf{G}}&=\sum\limits_{j=1}^{2^n}\mathbf{A}_j\overline{\big(\mathbf{F}+\mathbf{G}\big)}_j=\frac{1}{4}\sum\limits_{j=1}^{2^n}\big(\omega^{t_j}+\omega^{k_j}\big)\big(\overline{\omega^{t_j}}-\overline{\omega^{k_j}}\big)\\
		&=\frac{1}{4}\sum\limits_{j=1}^{2^n}\big(1-1+\omega^{k_j}\overline{\omega^{t_j}}-\omega^{t_j}\overline{\omega^{k_j}}\big)=\frac{1}{2}\mathrm{Im}\left(\sum\limits_{j=1}^{2^n}\omega^{k_j}\overline{\omega^{t_j}}\right)i,\\
		\overline{\scal{\mathbf{A}}{\mathbf{F}+\mathbf{G}}}&=-\frac{1}{2}\mathrm{Im}\left(\sum\limits_{j=1}^{2^n}\omega^{k_j}\overline{\omega^{t_j}}\right)i,
	\end{align*}	
	\begin{align*}
		\scal{\mathbf{B}}{\mathbf{F}-\mathbf{G}}&=\sum\limits_{j=1}^{2^n}\mathbf{B}_j\overline{\big(\mathbf{F}-\mathbf{G}\big)}_j=\frac{1}{4}\sum\limits_{j=1}^{2^n}\big(\omega^{l_j}-\omega^{r_j}\big)\big(\overline{\omega^{l_j}}+\overline{\omega^{r_j}}\big)\\
		&=\frac{1}{4}\sum\limits_{j=1}^{2^n}\big(1-1+\omega^{l_j}\overline{\omega^{r_j}}-\omega^{r_j}\overline{\omega^{l_j}}\big)=\frac{1}{2}\mathrm{Im}\left(\sum\limits_{j=1}^{2^n}\omega^{l_j}\overline{\omega^{r_j}}\right)i,\\
		\overline{\scal{\mathbf{B}}{\mathbf{F}-\mathbf{G}}}&=-\frac{1}{2}\mathrm{Im}\left(\sum\limits_{j=1}^{2^n}\omega^{l_j}\overline{\omega^{r_j}}\right)i,
	\end{align*}
	therefore,
	both~$(\ref{equation:the first inner product})$ and~$(\ref{equation:the second inner product})$ are zero numbers.
\end{proof}

\section{Properties of self-dual gbent function}\label{section:Properties of self-dual gbent function}

\subsection{Upper bound for the number of self-dual gbent functions}

Let $2^{h-1}<q\leqslant2^h$. For any $f\in\mathcal{GF}_n^q$ it is possible to associate a unique sequence of Boolean functions $a_0,a_1,\ldots,a_{h-1}\in\mathcal{F}_n$ such that~\cite{Generalized_bent}
\begin{equation*}
	f(x)=a_0(x)+2a_1(x)+\ldots+2^{h-1}a_{h-1}(x),\ \ x\in\mathbb{F}_2^n.
\end{equation*}
In paper~\cite{Hodzic} it was proved that for the case $q=2^k$ and even $n$, provided that $f$ is gbent its dual gbent $\widetilde{f}$ has the following form
\begin{equation*}
	\widetilde{f}(x)=b_0(x)+2b_1(x)+\ldots+2^{h-1}b_{h-1}(x),\ \ x\in\mathbb{F}_2^n,
\end{equation*}
where $b_{k-1}=\widetilde{b_{k-1}}$ and the dual of  $b_j=\widetilde{b_{k-1}}\oplus\big(\widetilde{b_{k-1}\oplus b_j}\big)$. If $f$ is self-dual gbent then $b_{k-1}$ is self-dual Boolean function and for $j=0,1,\ldots,k-1$ Boolean functions ${\big(b_{k-1}\oplus b_j\big)}$ are self-dual. It follows the statement
\begin{proposition}
	It holds $\big\lvert\mathrm{SB}_{2^k}^+(n)\big\rvert\leqslant\big\lvert\mathrm{SB}_2^+(n)\big\rvert^k$.
\end{proposition}
Note that this bound is consistent with the results from the work~\cite{Decomposition}.

\subsection{Affinity of self-dual gbent function}\label{section:Affinity of self-dual gbent function}

In paper~\cite{Singh} for the case when $q$ is divisible by $4$, necessary and sufficient conditions for the bentness of generalized Boolean functions of the form 
\begin{equation*}
	f(x)=\sum\limits_{i=1}^n\lambda_ix_i+\lambda_0,
\end{equation*}
where $\lambda_0,\lambda_1,\ldots,\lambda_n\in\mathbb{Z}_q$, were obtained. Functions from this class are referred to as {\it affine} functions.

It is well known that Boolean bent function and, as a consequence, self-dual Boolean bent function can not be affine. The next result shows the non-existence of self-dual gbent functions within the class of affine functions.
\begin{theorem}
	There are no self-dual generalized bent functions in $n$ variables of the form
	\begin{equation*}
		f(x)=\sum\limits_{i=1}^n\lambda_ix_i+\lambda_0,
	\end{equation*}
	where $\lambda_0,\lambda_1,\ldots,\lambda_n\in\mathbb{Z}_q$.
\end{theorem}

\begin{proof}
	Let $f$ be an affine gbent function in $n$ variables (for the case $q$ not divisible by $4$ if such exists, otherwise the result follows), namely
	\begin{equation*}
		f(x)=\sum\limits_{i=1}^n\lambda_ix_i+\lambda_0,\ \ x\in\mathbb{F}_2^n,
	\end{equation*}
	where $\lambda_0,\lambda_1,\ldots,\lambda_n\in\mathbb{Z}_q$. It is self-dual if and only if
	\begin{align*}
		H_f(y)&=\sum\limits_{x\in\mathbb{F}_2^n}\omega^{f(x)}(-1)^{\langle x,y\rangle}=\omega^{\lambda_0}\sum\limits_{x\in\mathbb{F}_2^n}\omega^{\sum\limits_{i=1}^n\lambda_ix_i+\frac{q}{2}\langle x,y\rangle}\\
		&=\omega^{\lambda_0}\prod\limits_{i=1}^n\sum\limits_{x_i\in\mathbb{F}_2}\omega^{\lambda_ix_i+\frac{q}{2}y_ix_i}=\omega^{\lambda_0}\prod\limits_{i=1}^n\left(1+\omega^{\frac{q}{2}y_i+\lambda_i}\right),
	\end{align*}
	for any $y\in\mathbb{F}_2^n$.
	
	For every $y\in\mathbb{F}_2^n$ denote 
	\begin{align*}
		\widehat{y}&=\left(y_1,y_2,\ldots,y_{n-1}\right)\in\mathbb{F}_2^{n-1},\\
		P_{n-1}\left(\widehat{y}\right)&=\left(1+\omega^{\frac{q}{2}y_1+\lambda_1}\right)\left(1+\omega^{\frac{q}{2}y_2+\lambda_2}\right)\cdot\ldots\cdot\left(1+\omega^{\frac{q}{2}y_{n-1}+\lambda_{n-1}}\right),\\
		a_{n-1}\left(\widehat{y}\right)&=\lambda_1y_1+\lambda_2y_2+\ldots+\lambda_{n-1}y_{n-1}.
	\end{align*}
	Then for any $y\in\mathbb{F}_2^n$ such that $y_n=0$ it holds
	\begin{equation*}
		P_{n-1}\left(\widehat{y}\right)\left(1+\omega^{\lambda_n}\right)=2^{n/2}\omega^{a_{n-1}\left(\widehat{y}\right)},
	\end{equation*}
	and for any $y\in\mathbb{F}_2^n$ such that $y_n=1$:
	\begin{equation*}
		P_{n-1}\left(\widehat{y}\right)\left(1+\omega^{\frac{q}{2}+\lambda_n}\right)=2^{n/2}\omega^{a_{n-1}\left(\widehat{y}\right)+\lambda_n}.
	\end{equation*}
	
	So, for any $\widehat{y}\in\mathbb{F}_2^{n-1}$ consider the system
	\begin{equation*}
		\begin{cases}
			P_{n-1}\left(\widehat{y}\right)\left(1+\omega^{\lambda_n}\right)=2^{n/2}\omega^{a_{n-1}\left(\widehat{y}\right)},\\
			P_{n-1}\left(\widehat{y}\right)\left(1-\omega^{\lambda_n}\right)=2^{n/2}\omega^{a_{n-1}\left(\widehat{y}\right)+\lambda_n}.
		\end{cases}
	\end{equation*}
	
	It is equivalent to
	\begin{equation*}
		\begin{cases}
			P_{n-1}\left(\widehat{y}\right)\left(1+\omega^{\lambda_n}\right)=2^{n/2}\omega^{a_{n-1}\left(\widehat{y}\right)},\\
			P_{n-1}\left(\widehat{y}\right)\left(1-\omega^{\lambda_n}\right)=P_{n-1}\left(\widehat{y}\right)\left(1+\omega^{\lambda_n}\right)\cdot\omega^{\lambda_n}.
		\end{cases}
	\end{equation*}
	
	Thus, we obtain the relation
	\begin{equation*}
		P_{n-1}\left(\widehat{y}\right)\left(1-\omega^{\lambda_n}\right)=P_{n-1}\left(\widehat{y}\right)\left(1+\omega^{\lambda_n}\right)\cdot\omega^{\lambda_n},
	\end{equation*}
	and can note that $P_{n-1}\left(\widehat{y}\right)\ne0$ since for any $y\in\mathbb{F}_2^n$ we have
	\begin{equation*}
		H_f(y)=\omega^{\lambda_0}P_{n-1}\left(\widehat{y}\right)\left(1+\omega^{\frac{q}{2}y_n+\lambda_n}\right),
	\end{equation*}
	and~$f$ is gbent that is $1-\omega^{\lambda_n}=\omega^{\lambda_n}+\left(\omega^{\lambda_n}\right)^2$. The solutions of this equation are~$\left(-1\pm\sqrt{2}\right)$. The norm of every of these numbers is not $1$ therefore $\omega^{\lambda_n}$ can not be a solution.
\end{proof}

\subsection{Self-dual gbent functions symmetric with respect to two variables}\label{section:Self-dual gbent functions symmetric with respect to two variables}

A generalized Boolean function $h\in\mathcal{GF}_{n+2}^q$ is said to be symmetric with respect to two variables $y$ and $z$ if there exist functions $f,g,s\in\mathcal{GF}_n^q$ such that
\begin{equation}\label{equation_symmetric}
	h(z,y,x)=f(x)+(y\oplus z)g(x)+yzs(x),\ \ y,z\in\mathbb{F}_2,x\in\mathbb{F}_2^n.
\end{equation}

In paper~\cite{Generalized_bent} it was proved that a function of such form is gbent if and only if the functions $f,f+g$ are gbent and $s(x)=q/2$, $x\in\mathbb{F}_2^n$. We study the conditions for self-duality of functions of such form.
\begin{theorem}
	Let $h$ be a gbent function of the form~$(\ref{equation_symmetric})$. Then $h$ is self-dual if and only if~$f$ is regular gbent, $g=\widetilde{f}+(q-1)f$, and $s(x)=q/2$, $x\in\mathbb{F}_2^n$.
\end{theorem}
\begin{proof}
	Let $F,FG$ be sign functions of regular gbent functions $f,f+g$. It is clear that
	\begin{equation*}
		\omega^h=
		\begin{pmatrix}
			F\\
			FG\\
			FG\\
			-F
		\end{pmatrix}.
	\end{equation*}
	Then the function $h$ is self-dual gbent if and only if
	\begin{equation*}
		\omega^{\widetilde{h}}=\mathcal{H}_{n+2}\omega^h=
		\frac{1}{2}\begin{pmatrix}
			\mathcal{H}_n& \mathcal{H}_n& \mathcal{H}_n& \mathcal{H}_n\\
			\mathcal{H}_n& -\mathcal{H}_n& \mathcal{H}_n& -\mathcal{H}_n\\
			\mathcal{H}_n& \mathcal{H}_n & -\mathcal{H}_n& -\mathcal{H}_n\\
			\mathcal{H}_n& -\mathcal{H}_n& -\mathcal{H}_n& \mathcal{H}_n
		\end{pmatrix}
		\begin{pmatrix}
			F\\
			FG\\
			FG\\
			-F
		\end{pmatrix}=
		\begin{pmatrix}
			F\\
			FG\\
			FG\\
			-F
		\end{pmatrix}
		=\omega^h.
	\end{equation*}
	
	Consider the system
	\begin{align*}
		\omega^{\widetilde{h}}&=
		\frac{1}{2}\begin{pmatrix}
			\mathcal{H}_n& \mathcal{H}_n& \mathcal{H}_n& \mathcal{H}_n\\
			\mathcal{H}_n& -\mathcal{H}_n& \mathcal{H}_n& -\mathcal{H}_n\\
			\mathcal{H}_n& \mathcal{H}_n & -\mathcal{H}_n& -\mathcal{H}_n\\
			\mathcal{H}_n& -\mathcal{H}_n& -\mathcal{H}_n& \mathcal{H}_n
		\end{pmatrix}
		\begin{pmatrix}
			F\\
			FG\\
			FG\\
			-F
		\end{pmatrix}
		\\
		&=
		\frac{1}{2}\begin{pmatrix}
			\mathcal{H}_nF+\mathcal{H}_n\left(FG\right)+\mathcal{H}_n\left(FG\right)+\mathcal{H}_nF\\
			\mathcal{H}_nF-\mathcal{H}_n\left(FG\right)+\mathcal{H}_n\left(FG\right)-\mathcal{H}_nF\\
			\mathcal{H}_nF+\mathcal{H}_n\left(FG\right)-\mathcal{H}_n\left(FG\right)-\mathcal{H}_nF\\
			\mathcal{H}_nF-\mathcal{H}_n\left(FG\right)-\mathcal{H}_n\left(FG\right)+\mathcal{H}_nF
		\end{pmatrix}	
		\\
		&=
		\frac{1}{2}\begin{pmatrix}
			\widetilde{F}-\widetilde{F}+2\widetilde{FG}\\
			2\widetilde{F}-\widetilde{FG}+\widetilde{FG}\\
			2\widetilde{F}+\widetilde{FG}-\widetilde{FG}\\
			-2\widetilde{FG}
		\end{pmatrix}
		=
		\begin{pmatrix}
			\widetilde{FG}\\
			\widetilde{F}\\
			\widetilde{F}\\
			-\widetilde{FG}
		\end{pmatrix}.
	\end{align*}
	Writing
	\begin{equation*}
		\begin{pmatrix}
			\widetilde{FG}\\
			\widetilde{F}\\
			\widetilde{F}\\
			-\widetilde{FG}
		\end{pmatrix}
		=
		\begin{pmatrix}
			F\\
			FG\\
			FG\\
			-F
		\end{pmatrix},
	\end{equation*}
	we see that $\widetilde{f}=f+g$, or, equivalently, $g=\widetilde{f}+(q-1)f$.
	
	Thus, we have
	\begin{equation*}
		h\left(z,y,x\right)=f(x)+\left(z\oplus y\right)\left[\widetilde{f}(x)+(q-1)f(x)\right]+\frac{q}{2}zy.
	\end{equation*}
\end{proof}

\section{Symmetries}\label{section:Symmetries}

In paper~\cite{Feulner} (see also~\cite{Self-dual}) it was shown that the mapping 
\begin{equation}\label{equation:preserve self-duality Boolean}
	f(x)\longrightarrow f\left(L\left(x\oplus c\right)\right)\oplus\left\langle c,x\right\rangle\oplus d,
\end{equation}
where $L\in\mathcal{O}_n$, $c\in\mathbb{F}_2^n$, $\mathrm{wt}(c)$ is even, $d\in\mathbb{F}_2$, preserves self-duality of a bent function. The group which consists of mappings of such form is called an {\it extended orthogonal group} and denoted by~$\overline{\mathcal{O}}_n$~\cite{Feulner}. It is known that this group is a subgroup of $\mathrm{GL}\left(n+2,\mathbb{F}_2\right)$~\cite{Feulner}.

In paper~\cite{Kutsenko_CC} known results were generalized within isometric mappings from the set of all mappings of all Boolean functions in $n\geqslant 4$ variables into itself, which preserve the Hamming distance. Namely it was proved the group of automorphisms of self-dual Boolean bent functions coincides with the extended orthogonal group.

In paper~\cite{Quaternary} it was proved that the mappings of the form
\begin{equation*}
	f(x)\longrightarrow f\left(Lx\right)+d,
\end{equation*}
where $L\in\mathcal{O}_n$, $d\in\mathbb{Z}_4$, preserve self-duality of a quaternary self-dual gbent function.

In current work we set the form~(\ref{equation:preserve self-duality Boolean}) for generalized case. The following result provides the construction of mappings of such form preserving the (anti-)self-duality of a Boolean function.
\begin{theorem}\label{theorem:preserve self-duality}
	The mapping of the set of all generalized Boolean functions in $n$ variables to itself of the form
	\begin{equation*}
		f(x)\longrightarrow f\left(L\left(x\oplus c\right)\right)+\frac{q}{2}\langle c,x\rangle+d,
	\end{equation*}
	where $L\in\mathcal{O}_n$, $c\in\mathbb{F}_2^n$, $\mathrm{wt}(c)$ is even, $d\in\mathbb{Z}_q$, preserves (anti-)self-duality of a gbent function.
\end{theorem}
\begin{proof}
	Let $f\in\mathrm{SB}_q^+(n)\cup\mathrm{SB}_q^-(n)$ that is $\widetilde{f}=f+\frac{q}{2}\varepsilon$ for some $\varepsilon\in\mathbb{F}_2$. Consider a function $g(x)=f\left(L\left(x\oplus c\right)\right)+\frac{q}{2}\langle c,x\rangle+d$, where $L\in\mathcal{O}_n$, $c\in\mathbb{F}_2^n$, $\mathrm{wt}(c)$ is even, $d\in\mathbb{Z}_q$. Its generalized Walsh--Hadamard transform is
	\begin{align*}
		H_g(y)&=\sum\limits_{x\in\mathbb{F}_2^n}{\omega^{g(x)}(-1)^{\left\langle x,y\right\rangle}}=\sum\limits_{x\in\mathbb{F}_2^n}{\omega^{f\left(L\left(x\oplus c\right)\right)+\frac{q}{2}\left\langle c,x\right\rangle+d+\frac{q}{2}\left\langle x,y\right\rangle}}\\
		&=\omega^d\sum\limits_{x\in\mathbb{F}_2^n}{\omega^{\frac{q}{2}\left\langle x,y\oplus c\right\rangle+f\left(L\left(x\oplus c\right)\right)}}=\omega^d\sum\limits_{z\in\mathbb{F}_2^n}{\omega^{\frac{q}{2}\left\langle L^{-1}z\oplus c,y\oplus c\right\rangle+f\left(z\right)}}\\
		&=\omega^{d+\frac{q}{2}\left\langle c,y\right\rangle+\frac{q}{2}\left\langle c,c\right\rangle}\sum\limits_{z\in\mathbb{F}_2^n}{\omega^{\frac{q}{2}\left\langle z,L\left(y\oplus c\right)\right\rangle+f\left(z\right)}}\\
		&=\omega^{d+\frac{q}{2}\left\langle c,y\right\rangle}2^{n/2}\omega^{\widetilde{f}\left(L\left(y\oplus c\right)\right)}=2^{n/2}\omega^{f\left(L\left(y\oplus c\right)\right)+\frac{q}{2}\langle c,y\rangle+d+\frac{q}{2}\varepsilon}\\
		&=2^{n/2}\omega^{g(y)+\frac{q}{2}\varepsilon}=2^{n/2}\omega^{\widetilde{g}(y)},
	\end{align*}
	hence $\widetilde{g}(y)=g(y)+\frac{q}{2}\varepsilon$ for any $y\in\mathbb{F}_2^n$.
\end{proof}

By using the mappings of this form we can clarify, for instance, the classification of quaternary self-dual bent functions in~$4$ variables given in~\cite{Quaternary} and formed by~$8$ classes. Namely, the representatives with vectors of values~$(0330302132010110)$ and~$(3123231322030300)$ from the classes~$4$ and~$5$ respectively are related by the transformation
\begin{equation*}
	f(x)\longrightarrow f(L(x\oplus c))+\frac{q}{2}\langle c,x\rangle+d,
\end{equation*}
where
\begin{equation*}
	L=\begin{pmatrix}
		1&0&0&0\\
		0&1&0&0\\
		0&0&1&0\\
		0&0&0&1
	\end{pmatrix},\ \ c=(1001),\ \ d=3.
\end{equation*}
The representatives with vectors of values~$(2022220222020200)$ and~$(2123230332121210)$ from the classes~$2$ and~$7$ respectively are related by the transformation
\begin{equation*}
	f(x)\longrightarrow f(L(x\oplus c))+\frac{q}{2}\langle c,x\rangle+d,
\end{equation*}
where
\begin{equation*}
	L=\begin{pmatrix}
		0&1&0&0\\
		1&0&0&0\\
		0&0&1&0\\
		0&0&0&1
	\end{pmatrix},\ \ c=(0101),\ \ d=1.
\end{equation*}

Thus, the classification of quaternary self-dual bent functions in~$4$ variables is given in the Table~\ref{table:classification}.

By a slight change of the parameters mentioned in Theorem~\ref{theorem:preserve self-duality}, it is possible to obtain the class of mapping that define a bijection between self-dual and anti-self-dual gbent functions in~$n$ variables.
\begin{table}
	\caption{Classification of quaternary self-dual bent functions in~$4$ variables}
	\label{table:classification}       
	\begin{tabular}{ll}
		\hline\noalign{\smallskip}
		Representative from equivalence class & Size\\ 
		\noalign{\smallskip}
		$0220202022000000$ & $24$\\ \noalign{\smallskip}
		$2022220222020200$& $64$ \\ \noalign{\smallskip}
		$0330313133110110$& $48$ \\ \noalign{\smallskip}
		$0330302132010110$& $120$ \\ \noalign{\smallskip}
		$1321213122010100$& $96$\\ \noalign{\smallskip}
		$0220213023100000$& $48$\\ 
		\noalign{\smallskip}
		Number of quaternary self-dual bent functions in four variables& $400$\\
		\noalign{\smallskip}\hline
	\end{tabular}
\end{table}

\begin{proposition}
	The mapping of the set of all generalized Boolean functions in $n$ variables to itself of the form
	\begin{equation*}
		f(x)\longrightarrow f(\pi(x))+g(x),
	\end{equation*}
	with
	\begin{equation*}
		\pi(x)=L\left(x\oplus c\right),\ \
		g(x)=\frac{q}{2}\langle c,x\rangle+d,\ \ x\in\mathbb{F}_2^n,
	\end{equation*}
	where $L\in\mathcal{O}_n$, $c\in\mathbb{F}_2^n$, $\mathrm{wt}(c)$ is odd, $d\in\mathbb{Z}_q$, is a bijection between the sets $\mathrm{SB}_q^+(n)$ and $\mathrm{SB}_q^-(n)$.
\end{proposition}
\begin{proof}
	Let $f\in\mathrm{SB}_q^+(n)\cup\mathrm{SB}_q^-(n)$ and $\widetilde{f}=f+\frac{q}{2}\varepsilon$ for some $\varepsilon\in\mathbb{F}_2$. One can show that for $g(x)=f\left(L\left(x\oplus c\right)\right)+\frac{q}{2}\langle c,x\rangle+d$, where $L\in\mathcal{O}_n$, $c\in\mathbb{F}_2^n$, $\mathrm{wt}(c)$ is odd,~$d\in\mathbb{Z}_q$, it holds $H_g(y)=2^{n/2}\omega^{\widetilde{g}(y)+q/2}$, $y\in\mathbb{F}_2^n$.
\end{proof}

From the existence of such bijections it follows that the cardinalities of the sets of self-dual and anti-self-dual gbent functions coincide.
\begin{corollary}
	It holds $\left\lvert\mathrm{SB}_q^+(n)\right\rvert=\left\lvert\mathrm{SB}_q^-(n)\right\rvert$.
\end{corollary}

\section{Conclusion}\label{section:Conclusion}

In current paper self-dual generalized bent functions were explored. A group of primary and secondary constructions was presented. The general form of self-dual Maiorana--McFarland gbent functions and their metrical properties were studied. The non-existence of affine self-dual gbent functions was shown. We also gave the description of self-dual gbent functions symmetric with respect to two variables. The properties of sign functions of self-dual gbent functions were considered. 

It is interesting to find other symmetries, if any, distinct from the ones that were found in this work. It involves the study of the automorphisms gropus of the considered gbent functions with respect to Hamming or Lee metrics. The study of connection with self-dual Boolean functions also seems to be a promising task.



\begin{thebibliography}{}
\bibitem{Carlet} Carlet~C.: Boolean functions for cryptography and error correcting codes. In: Crama Y., Hammer~P.L. (eds.) Boolean Models and Methods in Mathematics, Computer Science, and Engineering. p.~257--397. Cambridge University Press, Cambridge (2010)

\bibitem{Self-dual} Carlet~C., Danielsen~L.E., Parker~M.G., Sol\'e~P.: Self-dual bent functions, Int. J. Inform. Coding Theory, {\bf1}, 384--399 (2010)


\bibitem{Carlet_book} Carlet~C.: Boolean Functions for Cryptography and Coding Theory, 620 p., Cambridge University Press (2020)

\bibitem{Cesmelioglu} \c{C}e\c{s}melio\u{g}lu A., Meidl W. Pott A.: On the dual of (non)-weakly regular bent functions and self-dual bent functions. Adv. Math. Commun. {\bf7}(4), 425--440 (2013)

\bibitem{Cusick}
Cusick~T.W., St\u{a}nic\u{a}~P.: Cryptographic Boolean functions and applications, 288 p., Acad. Press, London, (2017)



\bibitem{Feulner} Feulner~T., Sok~L., Sol\'e~P., Wassermann~A.: Towards the Classification of Self-Dual Bent Functions in Eight Variables. Des. Codes Cryptogr. {\bf68}(1), 395--406 (2013)

\bibitem{Hou} Hou~X.-D.: Classification of self dual quadratic bent functions, Des. Codes Cryptogr. {\bf63}(2), 183--198 (2012)

\bibitem{Hou_generalized} Hou~X.-D.: Classification of $p$-ary self dual quadratic bent functions, $p$ odd. Journal of Algebra {\bf391}, 62--81 (2013)

\bibitem{Hyun} Hyun~J.Y., Lee~H., Lee~Y.: MacWilliams duality and Gleason-type theorem on self-dual bent functions. Des. Codes Cryptogr. {\bf63}(3), 295--304 (2012)

\bibitem{Janusz} Janusz~G.J.: Parametrization of self-dual codes by orthogonal matrices. Finite Fields Appl.~{\bf13}(3), 450--491 (2007)

\bibitem{HodzicConstructions} Hod\v{z}i\'{c}~S., Pasalic~E.: Generalized Bent Functions~--- Some General Construction Methods and Related Necessary and Sufficient Conditions. Cryptogr. Commun. {\bf7}(4), 469--483 (2015)

\bibitem{Hodzic2018} Hod\v{z}i\'{c}~S., Pasalic~E.: Construction methods for generalized bent functions. Discrete Appl. Math. {\bf238}, 14--23 (2018)

\bibitem{Hodzic} Hod\v{z}i\'{c}~S., Meidl~W., Pasalic~E.: Full Characterization of Generalized Bent Functions as (Semi)-Bent Spaces, Their Dual, and the Gray Image. IEEE Trans. Inform. Theory {\bf64}(7), 5432--5440 (2018)

\bibitem{Kumar} Kumar~P.V., Scholtz~R.A., Welch~L.R.: Generalized bent functions and their properties. J. Comb. Theory Series A {\bf40}, 90--107 (1985)

\bibitem{Kutsenko_MM} Kutsenko~A.V.: The Hamming Distance Spectrum Between Self-Dual
Maiorana--McFarland Bent Functions. Journal of Applied and Industrial Math. {\bf12}(1), 112--125 (2018)

\bibitem{Kutsenko_DCC} Kutsenko~A.: Metrical properties of self-dual bent functions. Des. Codes Cryptogr. {\bf88}(1), 201--222 (2020)

\bibitem{Kutsenko_CC} Kutsenko~A.: The group of automorphisms of the set of self-dual bent functions. Cryptogr. Commun. {\bf 12}(5), 881--898 (2020)

\bibitem{Kutsenko_PDM} Kutsenko~A., Tokareva~N.: Metrical properties of the set of bent functions in view of duality. Applied Discrete Math. №49, 18--34 (2020)

\bibitem{Luo} Luo~G., Cao~X., Mesnager~S.: Several new classes of self-dual bent functions derived from involutions. Cryptogr. Commun. {\bf11}(6), 1261--1273 (2019)

\bibitem{Decomposition} Sok~L., Shi~M., Sol\'e.~P.: Decomposition of bent generalized Boolean functions.  https://arxiv.org/abs/1611.06357v1.




\bibitem{Martinsen} Martinsen~T., Meidl~W., Stănică~P.: Partial spread and vectorial generalized bent functions. Des. Codes Cryptogr. {\bf 85}(1), 1--13 (2017)

\bibitem{McFarland} McFarland~R.L., A family of difference sets in non-cyclic groups. J. Combin. Theory. Ser. A {\bf 15}(1), 1--10 (1973)

\bibitem{Mesnager} Mesnager~S.: Several New Infinite Families of Bent Functions and Their Duals. IEEE Trans. Inf. Theory {\bf60}(7), 4397--4407 (2014)

\bibitem{Mesnager_book} Mesnager~S.: Bent Functions: Fundamentals and Results, 544 p., Springer, Berlin (2016)

\bibitem{Mesnager2018} Mesnager~S., Tang~C., Qi~Y., Wang~L., Wu~B., Feng~K.: Further Results on Generalized Bent Functions and Their Complete Characterization. IEEE Trans. Inform. Theory {\bf 64}(7), 5441--5452 (2018)

\bibitem{Paterson} Paterson~K.G. Generalized Reed–Muller Codes and Power Control in OFDM Modulation. IEEE Trans. Inform. Theory {\bf 46}(1), 104--120, (2000)

\bibitem{Preneel} Preneel~B., Van Leekwijck~W., Van Linden~L., Govaerts~R., Vandewalle~J.: Propagation characteristics of Boolean functions. In: Advances in Cryptology-EUROCRYPT. Lecture Notes in Computer Science, {\bf473}, pp. 161--173. Springer, Berlin (1990)

\bibitem{Riera} Riera~C., St\u{a}nic\u{a}~P., Gangopadhyay~S.: Generalized bent Boolean functions and strongly regular Cayley graphs. Discrete Appl. Math. {\bf283}, 367--374 (2020)

\bibitem{Rothaus} Rothaus~O.S.: On bent functions. J. Combin. Theory. Ser. A {\bf 20}(3), 300--305 (1976)

\bibitem{Schmidt} Schmidt~K.-U.: Quaternary constant-amplitude codes for multicode CDMA. IEEE Trans. Inform. Theory {\bf 55}(4), 1824--1832 (2009)


\bibitem{Singh} Singh~B.K.: On cross-correlation spectrum of generalized bent functions in generalized Maiorana--McFarland class. Information Sciences Letters {\bf 2}(3), 139--145 (2013)

\bibitem{Quaternary} Sok~L., Shi~M., Sol\'e.~P.: Classification and Construction of quaternary self-dual bent functions. Cryptogr. Commun. {\bf10}(2), 277--289 (2018)

\bibitem{Solodovnikov} Solodovnikov~V.I.: Bent functions from a finite Abelian group into a finite Abelian group. Discret. Math. Appl. {\bf12}(2), 111--126 (2002)

\bibitem{Generalized_bent} St\u{a}nic\u{a}~P., Martinsen~T., Gangopadhyay~S., Singh~B.~K.: Bent and generalized bent Boolean functions. Des. Codes Cryptogr., {\bf 69}(1), 77--94 (2013)

\bibitem{Tang} Tang~C., Xiang~C., Qi~Y., Feng~K.: Complete Characterization of Generalized Bent and~$2^k$-Bent Boolean Functions. IEEE Trans. Inform. Theory {\bf63}(7), 4668--4674 (2017)

\bibitem{Tokareva_survey} Tokareva~N.N.: Generalizations of bent functions~--- a survey. J. Appl. Ind. Math. {\bf 5}(1), 110--129 (2011)

\bibitem{Tokareva_iterative} Tokareva\;N.N.: On the number of bent functions from iterative constructions: lower bounds. Adv. Math. Commun. {\bf5}(4), 609--621 (2011)	

\bibitem{Tokareva_book} Tokareva~N.: Bent Functions, Results and Applications to Cryptography, 230 p., Acad. Press. Elsevier (2015)

\bibitem{Wada} Wada~T.: Characteristic bit sequences applicable to constant amplitude orthogonal multicode systems. IEICE Trans. Fundamentals {\bf E83-A}(11), 2160--2164, (2000)


\end{thebibliography}


\end{document}